\documentclass[11pt]{amsart}
\usepackage[a4paper,hmargin={2.5cm,2.5cm},vmargin={2.5cm,2.5cm}]{geometry}

\usepackage{amssymb,latexsym}

\newcommand{\ev}{\mathrm{ev}}
\def\twoheaddownarrow{\rlap{$\downarrow$}\raise-.5ex\hbox{$\downarrow$}}
\def\twoheaduparrow{\rlap{$\uparrow$}\raise.5ex\hbox{$\uparrow$}}

\newcommand{\uup}{\twoheaduparrow}

\newcommand{\JCocts}[2]{J\text{-}\mathbf{Cocts}(#1,#2)}
\newcommand{\Ord}{\mathbf{Ord}}
\newcommand{\Met}{\mathbf{Met}}
\newcommand{\UMet}{\mathbf{UMet}}
\newcommand{\colim}{\mathrm{colim}}
\newcommand{\QDOM}{(J,\mathcal{Q})\text{-}\mathbf{Dom}}
\newcommand{\QDOMindex}{(J,\mathcal{Q})\text{-}\mathbf{Dom}}
\newcommand{\Cat}[1]{#1\text{-}\mathbf{Cat}}
\newcommand{\Mod}[1]{#1\text{-}\mathbf{Mod}}
\newcommand{\Cocts}[1]{#1\text{-}\mathbf{Cocts}}

\newcommand{\V}{\mathcal{Q}}
\newcommand{\unit}{\mathbf{1}}
\newcommand{\tensor}{\otimes}
\newcommand{\bv}{\bigvee}

\newcommand{\Sup}{\textsf{S}}
\newcommand{\yoneda}{\textsf{y}}
\newcommand{\two}{\mathbf{2}}

\newcommand{\mmminus}[2]{#2 - #1}

\newcommand{\way}{\mathord{\Downarrow}}

\newcommand{\modto}{{\longrightarrow\hspace*{-2.8ex}{\circ}\hspace*{1.2ex}}}
\newcommand{\mate}[1]{\,^\ulcorner\! #1^\urcorner}

\newcommand{\blackright}{\mbox{ $-\!\mbox{\footnotesize $\bullet$}$ }}
\newcommand{\blackleft}{\mbox{ $\mbox{\footnotesize $\bullet$}\!-$ }}

\newcommand{\Idl}{\mathbf{Idl}}
\newcommand{\FC}{\mathbf{FC}}
\newcommand{\FSW}{\mathbf{FSW}}

\renewcommand{\implies}{\Rightarrow}

\newcommand{\pzb}{\subseteq}
\newcommand{\Nat}{\mathbb{N}}

\newcommand{\wb}{\ll}

\newcommand{\eps}{\varepsilon}
\newcommand{\op}{\mathrm{op}}

\newtheorem{theorem}{Theorem}[section]

\newtheorem{lemma}[theorem]{Lemma}

\theoremstyle{definition}
\newtheorem{definition}[theorem]{Definition}

\newtheorem{example}[theorem]{Example}

\numberwithin{equation}{section}

\begin{document}

\begin{abstract}
We describe a duality for quantale-enriched categories that extends the Lawson
duality for continuous dcpos: for any saturated class $J$ of modules that commute with certain weighted limits, 
and under an appropriate choice  of morphisms, the category of $J$-cocomplete and $J$-continuous quantale-enriched categories is self-dual.
\end{abstract}


\title{A duality of quantale-enriched categories}

\author{Dirk Hofmann}
\thanks{The first author acknowledges partial financial assistance by Centro de Investiga\c{c}\~{a}o e Desenvolvimento Matem\'{a}tica e Aplica\c{c}\~{o}es da Universidade de Aveiro/FCT and the project MONDRIAN (under the contract PTDC/EIA-CCO/108302/2008).}
\address{Departamento de Matem\'atica, Universidade de Aveiro, 3810-193 Aveiro, Portugal}
\email{dirk@ua.pt}

\author{Pawe\l\ Waszkiewicz}
\thanks{The second author is supported  by grant number N206 3761 37 funded by Polish Ministry of Science and Higher Education.}
\address{Theoretical Computer Science, Jagiellonian University, ul.Prof.S.\L ojasiewicza 6, 30-348 Krak\'{o}w, Poland}
\email{pqw@tcs.uj.edu.pl}

\maketitle

\section{Introduction}
In \cite{hw10} we observed that the left adjoint to the Yoneda embedding in a
quantale-enriched category
$X$ can be interpreted as a notion of approximation in $X$. Thus in
directed-complete posets, approximation is the way-below 
relation \cite{compendium}.I.1.; in complete lattices the totally-below
relation \cite{ran53}; and in (generalised) metric spaces a distance 
$\way\colon X\times X\to[0,\infty]$ such that every $x\in X$ is a ``metric supremum'' of
$\way(-,x)$ \cite{hw10}.

The purpose of this paper is to develop a duality theory for $\V$-categories
that extends the Lawson duality for continuous  dcpos \cite{law79}. Recall that Lawson's 
theorem states that the category of  continuous dcpos with Scott-open filter reflecting maps is self-dual. We show that
under an appropriate choice of morphisms the category of $J$-cocomplete and
$J$-continuous \mbox{($=$ admitting approximation)} \mbox{$\V$-categories} is
self-dual. Our duality theorem holds for any
saturated class $J$ of modules that preserve certain limits; 
therefore it works uniformly for continuous domains, completely distributive
complete lattices, Yoneda-complete quasi-metric spaces, totally distributive $\V$-categories, and
perhaps many other familiar structures from the borderline of metric and order theory.

Our feet rest on shoulders of many. Hausdorff's point of view that a metric is a~relation valued in
non-negative real numbers, brought to light by \cite{law73}, led to a~development of an unified categorical/algebraic description 
of topology, uniformity, order and metric \cite{CH03, CT03, onesetting}.  The idea of relative cocompleteness was developed 
in \cite{Kel_EnrCat, albert:kelly, kock95, KS_Colim, KL_MonCol, schmitt:flat}. Our primary 
examples of classes of modules have already been studied in \cite{FSW, schmitt:flat, vickers}. We do hope that our 
results will be of interest to those who work with categories where the left adjoint to Yoneda embedding has a~left adjoint;
research in this direction include: \cite{joh82, kos86, FW90, RW91, stu07}.

\section{Preliminaries}

\subsection{Quantales}

A $\V = (Q,\leqslant,\otimes,\unit)$ is a commutative unital quantale (in short: a~quantale) such that the unit element $\unit$ is greatest with respect to the order on $(Q, \leqslant)$. We also assume that $\bot\neq \unit$. Examples of quantales include: the two element lattice $\two = (\{\bot,\unit\},\leqslant, \wedge,\unit)$; the unit interval $[0,1]$ in the natural order, with multiplication as tensor; the extended real half line $[0,\infty]$ in the order opposite to the natural one, with addition as tensor. In general, every Heyting algebra with infimum as tensor is a quantale.

\subsection{$\V$-categories}
We recall that a \emph{$\V$-category} is a set $X$ with a map $X\colon X\times X\to Q$, called {\em the structure of} $X$, with two properties: $\unit\leqslant X(x,x)$ for all $x\in X$ (reflexivity), and $X(x,y)\otimes X(y,z)\leqslant X(x,z)$ for all $x,y,z\in X$ (transitivity). In our paper $\Cat{\V}$ denotes the category of $\V$-categories, where morphisms, called $\V$-functors, are maps $f\colon X\to Y$ such that $X(x,z)\leqslant Y(fx,fz)$ for all $x,z\in X$. For example $\Met := \Cat{[0,\infty]}$ is Lawvere's category of generalised metric spaces \cite{law73}, where reflexivity and transitivity correspond respectively to the assumption of self-distance being zero and to the triangle inequality. As another example we consider $\Cat{\two}$, which is isomorphic to the category of preordered sets and monotone maps, and will henceforth be denoted by $\Ord$.\\

\noindent
A $\V$-category is separated if $X(x,y)=X(y,x)=\unit$ implies $x=y$, for all $x,y\in X$. For example a separated $[0,\infty]$-category is a quasi-metric space, where points can possibly be at infinite distance. Any $\V$-category $X$ is preordered by the relation $x\leqslant_X y$ iff $\unit\leqslant X(x,y)$, which is antisymmetric iff $X$ is separated. Clearly, $\V$-functors are $\leqslant_X$-preserving.\\

\noindent The internal hom of $\Cat{\V}$ is the set $Y^X$ of all $\V$-functors of type $X\to Y$ considered with the structure $Y^X(f,g) := \bigwedge_{x\in X} Y(fx,gx)$. The induced order on $Y^X$ is pointwise. The quantale $\V$ is made into a separated
$\V$-category by its internal hom. The induced order $\leqslant_{\V}$ coincides with the original order on $Q$. By $X^\op$ we mean the $\V$-category dual to $X$. $\widehat{X}$ is defined as $\V^{X^\op}$, that is $\widehat{X}(f,g) = \bigwedge_{x\in X} \V(fx, gx)$. For any $X$, we have the $\V$-functor $\yoneda_X\colon X\to\widehat{X}$, $\yoneda_X x = X(-,x)$, called the {\em Yoneda embedding}. The Yoneda embedding is fully faithful. Furthermore, for all $x\in X$ and $f\in\widehat{X}$, we have $\widehat{X}(\yoneda_X x,f)=fx$, and this equality is the statement of the Yoneda Lemma for $\V$-categories.\\

\noindent Lastly, $\Cat{\V}$ admits a tensor product $X\otimes Y((x,y),(z,w))=X(x,z)\otimes Y(y,w)$. Since tensor is left adjoint to internal hom, every $\V$-functor $g\colon X\otimes Y\to Z$ has its {\em exponential mate} $\mate{g}\colon Y\to Z^X$. It is worth noting that the structure of $X$ is always a $\V$-functor of type $X^\op \otimes X\to Q$, and its exponential mate is the Yoneda embedding $\yoneda_X:X\to\widehat{X}$.

\subsection{$\V$-modules}
A $\V$-functor of type $X^\op \otimes Y\to Q$ is called a {\em $\V$-module} (or plainly: a {\em module}). For example, the structure of any $\V$-category $X$ is a module. Moreover, any two modules $\phi\colon X^\op\otimes Y\to Q$ and $\psi\colon Y^\op \otimes Z\to Q$ can be composed to give a module of type $X^\op \otimes Z\to Q$:
\[(\psi \cdot \phi) (x,z) := \bigvee_{y\in Y} (\phi(x,y)\otimes \psi(y,z)).\]
Therefore we think of $\phi\colon X^\op\otimes Y\to Q$ as an arrow $\phi\colon X\modto Y$, which, by the above, can be composed with
$\psi\colon Y\modto Z$ to give $\psi\cdot \phi\colon X\modto Z$. Note
also that $Y\cdot \phi = \phi = \phi\cdot X$.\\

\noindent Any function $f\colon X\to Y$ gives rise to two
modules, namely $f_*\colon X\modto Y$, $f_*(x,y) = Y(fx,y)$ and $f^*\colon Y\modto X$, $f^*(y,x) = Y(y,fx)$. We further observe that for any element $x\colon 1\to X$ ($1$ is the one-element $\V$-category that should not be confused with the unit of the quantale), the module $x^*\colon X\modto 1$ is in fact the same as the $\V$-functor $\yoneda_X x := X(-,x)\in \widehat{X}$. Dually, the module $x_*\colon 1\modto X$ corresponds to the $\V$-functor $\lambda_X x:=X(x,-)$.\\

\noindent The set of all modules of type $X\modto Y$ becomes a complete lattice via the pointwise order where the supremum $\phi$ of a family $\phi_i\colon X\modto Y$ ($i\in I$) of modules can be calculated as $\phi(x,y)=\bigvee_{i\in I}\phi_i(x,y)$. Furthermore, composition of modules preserves this suprema on both sides, and therefore the maps $-\cdot\phi$ and $\phi\cdot-$ have right adjoints $-\blackleft\phi$ and $\phi\blackright-$ respectively. Explicitly, given $\phi\colon X\modto Y$,
\[
(\psi\blackleft\phi)(y,z)=\bigwedge_{x\in X}\V(\phi(x,y),\psi(x,z)
\]
for any $\psi \colon X\modto Z$, and
\[
(\phi\blackright \psi)(z,x)=\bigwedge_{y\in Y}Q(\phi(x,y),\psi(z,y))
\]
for any $\psi \colon Z\modto Y$. We call $\psi\blackleft\phi$ the \emph{extension} of $\psi$ along $\phi$, and $\phi\blackright\psi$ the \emph{lifting} of $\psi$ along $\phi$. This construction will be used to define the so called \emph{way-below module} in Section~\ref{subsect:J-continuous}.\\

\noindent In $\Ord$, modules of type $X\modto 1$ are precisely
(characteristic maps of) lower sets, and modules of type
$1\modto X$ are upper sets of the poset $X$. Furthermore, the up-set of all upper bounds of $\psi\colon:X\modto 1$ is given by $\phi=(\leqslant\blackleft\psi)$, and $x\in X$ is a smallest upper bound of $\psi$ if and only if $x_*=(\leqslant\blackleft\psi)$. On the other hand, in $\Met$, any Cauchy sequence
$(x_n)_{n\in\omega}$ induces a module $\phi\colon 1\modto X$
via $\phi(x) = \lim_{n\to\infty} X(x_n,x)$, and a module
$\psi\colon X\modto 1$ via $\psi(x) = \lim_{n\to\infty} X(x,x_n)$.
Observe that $\psi\cdot \phi \leqslant 0$ and $\phi\cdot
\psi\geqslant X$ in the pointwise order. Conversely, any pair of
modules that satisfies the above equations comes from some
Cauchy sequence on $X$. More generally, we will say that modules $\phi\colon
Z\modto X$, $\psi \colon X\modto Z$ are adjoint iff $\phi\cdot
\psi\leqslant X$ and $\psi\cdot\phi\geqslant Z$. In this case we say
that $\phi$ is a left adjoint to $\psi$ and $\psi$ is a right
adjoint to $\phi$.

\subsection{$J$-cocomplete $\V$-categories}

We recall here briefly the notions of weighted limit and  weighted colimit, for further details we refer to \cite{Kel_EnrCat,KS_Colim}. For a module $\phi\colon 1\modto I$, a $\phi$-weighted limit of a $\V$-functor $h\colon I\to X$ is an element $x\in X$ with $x^*=\phi\blackright h^*$. Dually, for a module $\psi\colon I\modto 1$, a $\psi$-weighted colimit of a $\V$-functor $h\colon I\to X$ is an element $x\in X$ with $x_*=h_*\blackleft\psi$. A $\V$-category $X$ is called \emph{complete} if $X$ admits all weighted limits, and \emph{cocomplete} if $X$ admits all weighted colimits. For instance, $\V$ is both complete and cocomplete where the limit of $h$ and $\phi$ is given by $\bigwedge_{i\in I}\V(\phi(i),h(i))$ and the colimit of $h$ and $\psi$ by $\bigvee_{i\in I}\psi(i)\otimes h(i)$. This argument extends pointwise to $\widehat{X}$, and we also note that a $\V$-category $X$ is complete if and only if $X$ is cocomplete.\\

\noindent One says that a $\V$-functor $f:X\to Y$ preserves the $\phi$-weighted limit $x$ of $h\colon I\to X$ if $f(x)$ is a $\phi$-weighted limit of $f h:I\to Y$, likewise,  $f:X\to Y$ preserves the $\psi$-weighted colimit $x$ of $h\colon I\to X$ if $f(x)$ is a $\psi$-weighted colimit of $f h:I\to Y$. Then $f:X\to Y$ is called \emph{continuous} if $f$ preserves all existing weighted limits in $X$, and $f$ is called \emph{cocontinuous} if $f$ preserves all existing weighted colimits in $X$.\\

\noindent In the sequel we will be interested in special kinds of colimits, hence we suppose that there is given a collection $J$ of modules of type $X\modto 1$, called thereafter \emph{$J$-ideals}. The set of those modules in $J$ with domain $X$ we denote as $JX$. Then we define $X$ to be \emph{$J$-cocomplete} if $X$ admits all $\psi$-weighted colimits with $\psi$ in $J$, and a $\V$-functor $f:X\to Y$ is called \emph{$J$-cocontinuous} if $f$ preserves all existing $J$-weighted colimits in $X$. We will also assume that our class $J$ of modules is \emph{saturated}, which amounts to saying that $JX$ contains all modules $x^*:X\modto 1$ and is closed in $\widehat{X}$ under $J$-weighted colimits. In this case, $X$ is $J$-cocomplete if and only if $X$ admits all $\psi$-weighted colimits with $\psi:X\modto 1$ in $\Mod{J}$, which in turn is equivalent to $\yoneda_X\colon X\to JX$ having a left adjoint in $\Cat{\V}$. That is, there must exist a $\V$-functor $\Sup_X\colon JX\to X$ such that for all $\phi\in JX$ and all $x\in X$:
\begin{equation}
\label{eq:ideal}
X(\Sup_X\phi, x) = \widehat{X}(\phi,\yoneda_X x).
\end{equation}
The element $\Sup_X\phi\in X$ is called the \emph{supremum} of $\phi$. If $JX=\widehat{X}$ and $\Psi:\widehat{X}\modto 1$, then $\Sup_X(\Psi)(x)=\bigvee_{\psi\in\widehat{X}}\Psi(\psi)\otimes\psi(x)=\bigvee_{\psi\in\widehat{X}}\Psi(\psi)\otimes[\yoneda(x),\psi]$, hence $\Sup_X(\Psi)=\Psi\cdot\yoneda_*(x)$. Since $JX$ is closed in $\widehat{X}$ under $J$-colimits, the same formula describes $J$-suprema in $JX$. For example, if $\V=\two$, then $\widehat{X}$ is a poset of lower subsets of the poset $X$ ordered by inclusion, $\psi$ is a lower set of lower sets of $X$, and the supremum of $\psi$ is nothing else but
$\bigcup\psi$.\\

\noindent A $\V$-functor $f\colon X\to Y$ between $J$-cocomplete $\V$-categories is $J$-cocontinuous if and only if $f(\Sup\phi) = \Sup(
Jf(\phi))$, for all $\phi\in JX$. Here we make use of the fact that $J$ defines a functor $J:\Cat{\V}\to\Cocts{J}$ which sends a $\V$-category $X$ to $JX$, and a $\V$-functor $f:X\to Y$ to $Jf:JX\to JY,\,\psi\mapsto\psi\cdot f^*$. We use the occasion to remark that $J:\Cat{\V}\to\Cocts{J}$ is left adjoint to the inclusion functor $\Cocts{J}\to\Cat{\V}$. Even better, $\Cocts{J}\to\Cat{\V}$ is monadic which we need here only to conclude that $\Cocts{J}$ is complete and limits in $\Cocts{J}$ are calculated as in $\Cat{\V}$. For details we refer to \cite{KL_MonCol}.\\

\noindent There is a well-known general procedure to specify a saturated class $J$ of modules which we describe now. 
\begin{example}\label{ex:JdefinedByPhi}
Fix a collection $\Phi$ of modules $\phi:1\modto I$, and define $J$ as the class of all those modules $\psi:X\modto 1$ where the $\V$-functors
\[
\psi\cdot-:\V^X\to\V,\,\alpha\mapsto\psi\cdot\alpha=\bigvee_{x\in X}\alpha(x)\otimes\psi(x).
\]
preserve $\Phi$-weighted limits. Here we identify a $\V$-functor $\alpha:X\to\V$ with a module $\alpha:1\modto X$. Explicitly, we require that, for any $\phi:1\modto I$ in $\Phi$ and any $\V$-functor $\alpha_{-}:I\to\V^X$,
\[
\bigwedge_{i\in I}\V(\phi(i),\bigvee_{x\in X}\alpha_i(x)\otimes\psi(x))
=
\bigvee_{x\in X}\left(\bigwedge_{i\in I}\V(\phi(i),\alpha_i(x))\right) \otimes\psi(x).
\]
Note that $\V$-functoriality of $\psi\cdot-$ implies already that the left hand side is larger or equal to the right hand side.
\end{example}

\noindent Cocompleteness relative to $J$ allows for a unified presentation of seemingly unrelated notions of order- and metric completeness:

\begin{example}
\label{extremecocompl}
For any $\V$, there is a largest and a smallest choice of $J$: let either $J$ consist of all modules of type $X\modto 1$, or only of representable modules $x^*:X\modto 1$ where $x\in X$. In the first case a $\V$-category $X$ is $J$-cocomplete if and only if it is cocomplete, and in the second case every $\V$-category is $J$-cocomplete.
\end{example}

\begin{example}
\label{ex:ord} For $\V=\two$, we consider all modules of
type $X\modto 1$ corresponding to order-ideals in $X$ (i.e. directed
and lower subsets of $X$), and write $J= \Idl$. Then $X$ is
$\Idl$-cocomplete iff $X$ is a directed-complete.
\end{example}
\begin{example}
\label{ex:met} For $\V=[0,\infty]$ we consider all modules
of type $X\modto 1$ corresponding to ideals in $X$ in the sense of
\cite{rutten:yoneda}, and write $J=\FC$. These ideals in turn
correspond to equivalence classes of forward Cauchy sequences on
$X$. Hence, $X$ is $\FC$-cocomplete if and only if each forward
Cauchy sequence on $X$ converges if and only if $X$ is sequentially Yoneda
complete.
\end{example}

\begin{example}
For any $\V$ we can choose $J$ to consist of all right adjoint modules (i.e. modules that have left adjoints).
Recall from \cite{law73} that, for $\V = [0,\infty]$, a right adjoint module $X\modto 1$ corresponds to an equivalence class of Cauchy sequences on $X$. A generalised metric space $X$ is $J$-cocomplete if and only if each Cauchy sequence on $X$ converges.
\end{example}

\begin{example}
\label{ex:FSWideals} For a completely distributive quantale $\V$ with totally below relation $\prec$ and any $\V$-category $X$, a module $\psi\colon X\modto 1$ is a $\FSW$-ideal if: (a) $\bigvee_{z\in X} \psi z = \unit$, and (b)
for all $e_1,e_2,d\prec \unit$, for all $x_1,x_2\in X$, whenever
$e_1\prec \psi x_1$ and $e_2\prec \psi x_2$, then there exists $z\in
X$ such that $d\prec \psi z$, $e_1\prec X(x_1,z)$ and $e_2\prec
X(x_2,z)$. Now for $\V = [0,\infty]$ $\FSW$-ideals on $X$ are in a
bijective correspondence with equivalence classes of forward Cauchy
nets on $X$ \cite{FSW}; for $\V=\two$, $\FSW$-ideals are
characteristic maps of order-ideals on $X$. Therefore this example
unifies Examples \ref{ex:ord}, \ref{ex:met}.
\end{example}

\begin{example}\label{example:vickers}
For any quantale $\V$, a module $\psi\colon X\modto 1$ is called \emph{flat} if the map $(\psi\cdot -)$ taking modules of type $1\modto X$ to $\V$
preserves finite meets. For $\V=\two$, one verifies that $\psi:X\modto 1$ is flat if and only if $\psi:X^\op\to\two$ is the characteristic map of a directed down-set. For $\V=[0,\infty]$ with $\otimes=+$, Theorem 7.15 of \cite{vickers} states that flat modules are the same as $\FSW$-ideals, therefore this example unifies Examples \ref{ex:ord}, \ref{ex:met} as well. However, as we will show in Subsection \ref{subsect:ultramet}, flat modules and $\FSW$-ideals are in general different.
\end{example}

\begin{example}
\label{ex:formal balls}
For any $\V$, put $JX$ to be the set of all modules $\psi:X\modto 1$ of the form $\psi=u\cdot x^*$ where $x\in X$ and $u\in Q$. 
Here we think of $u\in Q$ as a module $1\modto 1$. Spelled out, for $y\in X$ one has $\psi(y)=X(y,x)\otimes u$. 
Note that $\psi(y)=\bot$ whenever $u=\bot$, independently of $x\in X$. A $\V$-category $X$ is $J$-cocomplete if it admits ``tensoring'' 
with elements of $\V$ in the following sense: for any $x\in X$ and $u\in Q$, there exists a (necessarily unique up to equivalence) element $z\in X$ with
\[
 X(z,y)=\V(u,X(x,y))
\]
for all $y\in X$, and one denotes $z$ as $u\otimes x$. 
\end{example}

\subsection{$J$-continuous $J$-cocomplete $\V$-categories}
\label{subsect:J-continuous}

$J$-continuity for $\V$-categories, introduced in \cite{hw10}, allows for a unified treatment of many structures that play a major role in theoretical computer science, e.g. continuous domains, complete metric spaces, or completely distributive complete lattices.

\begin{definition}
A $J$-cocomplete $\V$-category $X$ is \emph{$J$-continuous} if the supremum $\Sup_X\colon JX\to X$ has a left adjoint.
\end{definition}

Note that any $\V$-functor of type $X\to JX$ corresponds to a certain module $X\modto X$ belonging to $J$. Hence, $X$ is
$J$-continuous if and only if there exists a module $\way_X\colon X\modto X$ in $J$ with $\mate{\way_X}\dashv \Sup_X$. It is not difficult to see that $\Sup_X^*\cdot\way_X\leqslant\yoneda_{X*}$, and $\way_X$ is the largest module that satisfies this inequality; hence we have identified $\way_X\colon X\modto X$ as the lifting $\way_X = \Sup_X^* \blackright{\yoneda_{X}}_*$. In fact, module $\way_X:= \Sup_X^* \blackright{\yoneda_{X}}_*$ exists for any $J$-cocomplete $\V$-category, and we refer to it as the \emph{way-below} module. It is worth noting that $JX$ is $J$-continuous for every $\V$-category $X$. In this case, the way-below module is given by
\begin{equation}\label{way_for_JX}
\way(\psi,\psi')=\bigvee_{x\in X}\psi'(x)\otimes[\psi,x^*].
\end{equation}

In the simplest case, $\V=\two$ and $J=\Idl$, the module $\way_X$ is indeed the (characteristic map of the) way-below relation on $X$. In the case of metric spaces, as a~consequence of symmetry, $\way_X\colon X\modto X$ is the same as the structure $X\colon X\modto X$.

We call a module $v\colon X\modto X$ {\em auxiliary}, if $v\leqslant X$; {\em interpolative}, if $v\leqslant v\cdot v$; {\em approximating}, if $v\in J$ and $X\blackleft v = X$; {\em $J$-cocontinuous}, if $\Sup_X^*\cdot v = \yoneda_{X*}\cdot v$. In a $J$-continuous $J$-cocomplete $\V$-category, the way-below module is auxiliary, interpolative, approximating and $J$-cocontinuous. In fact, we show \cite{hw10} that a $J$-cocomplete $\V$-category is $J$-continuous iff the way-below module is approximating.

Consider some examples: $\FSW$-continuous $\FSW$-cocomplete $\two$-categories are precisely continuous domains; cocontinuous cocomplete $\two$-categories are completely distributive complete lattices (there the way-below module becomes the `totally-below' relation associated with complete distributivity of the underlying lattice); $[0,\infty]$ considered with the generalised metric structure $[0,\infty](x,y) = \max\{y-x,0\}$ is an $\FSW$-continuous $\FSW$-complete $[0,\infty]$-category; complete metric spaces are $\FSW$-continuous $\FSW$-cocomplete $[0,\infty]$-categories.


\subsection{Open modules}
\label{subsect:open modules}
Let $\JCocts{X}{Y}$ denote the set of all $J$-cocontinuous $\V$-functors from $X$ to $Y$, and we view $\JCocts{X}{\V}$ as a sub-$\V$-category of $\V^X$.

\begin{lemma}
$\JCocts{X}{\V}$ is closed under arbitrary suprema in $\V^X$. Hence, $\JCocts{X}{\V}$ is cocomplete.
\end{lemma}
\begin{proof}
Just observe that $\bigvee\colon\V^I\to\V$ is a $\V$-functor left adjoint to the diagonal $\Delta\colon\V\to\V^I$, for any set $I$; and $u\otimes-\colon\V\to\V$ is a $\V$-functor left adjoint to $\V(u,-)\colon\V\to\V$.
\end{proof}

From the lemma above we deduce that the inclusion functor $\JCocts{X}{\V}\hookrightarrow \V^X$ has a right adjoint $v\colon \V^X\to \JCocts{X}{\V}$.

If $X$ is $J$-cocomplete and $J$-continuous, this right adjoint has a simple description. In fact, since $\way_X\dashv \Sup_X$ and $\Sup_X\dashv \yoneda_X$, the map $\V^X\to \JCocts{X}{\V}$, $f\mapsto f_L\cdot \way_X$ (where $f_L$ is left Kan extension of $f$) is right adjoint to $\JCocts{X}{\V}\hookrightarrow \V^X$ in $\Ord$, hence it underlies $v$. Hence in this case we can write $v$ as the corestriction of the composite of left adjoints
\[\V^X \longrightarrow \JCocts{JX}{\V}\hookrightarrow \V^{JX}\stackrel{-\cdot \way_X}{-\!\!\!-\!\!\!-\!\!\!\longrightarrow} \V^X\]
to $\JCocts{X}{\V}$, hence $v$ is itself left-adjoint.

\begin{lemma}
If $X$ is $J$-cocomplete and $J$-continuous, then $\JCocts{X}{\V}$ is totally continuous.
\end{lemma}

\begin{proof}
$\V^X$ is totally continuous, and $\JCocts{X}{\V}$ inherits this property since $v\colon \V^X\to \JCocts{X}{\V}$ is a left and a right adjoint.
\end{proof}

We put now $FX := \JCocts{X}{\V}\cap J(X^\op)$ and call $\alpha\in FX$ an {\em open module}. More precisely, $FX$ is defined via the pullback in $\Cocts{J}$  of two inclusions: $\JCocts{X}{\V}\hookrightarrow \V^X$, $J(X^\op)\hookrightarrow \V^X$, which tells us that:
\begin{itemize}
\item $FX$ is $J$-cocomplete,
\item both inclusion maps $FX\hookrightarrow J(X^\op)$ and $FX\hookrightarrow \JCocts{X}{\V}$ preserve $J$-suprema.
\end{itemize}

\begin{definition}
We say that a $J$-continuous $\V$-category $X$ is \emph{open module determined} if for all $x,y\in X$:
\begin{equation}
\label{eq:technical lemma}
\way_X(x,y)= \bigvee_{\alpha\in FX}(\alpha(y)\tensor [\alpha, \lambda_X(x)]).
\end{equation}
\end{definition}
\noindent Note that, for all $\alpha\in FX$ and $x,y\in X$,
\[
\alpha(y)\tensor [\alpha, \lambda_X(x)]
=\bigvee_{z\in X}(\alpha(z)\otimes\way_X(z,y)\otimes[\alpha,X(x,-)])
\leqslant\bigvee_{z\in X}X(x,z)\otimes\way_X(z,y)=\way_X(x,y),
\]
hence \eqref{eq:technical lemma} is equivalent to
\begin{equation*}
\way_X(x,y)\leqslant \bigvee_{\alpha\in FX}(\alpha(y)\tensor [\alpha, \lambda_X(x)]).
\end{equation*}
Furthermore, \eqref{eq:technical lemma} is equivalent to
\[
\way_X(x,y)= \bigvee_{\alpha\in FX}(\alpha(y)\tensor [\alpha, \way_X(x,-)])
\]
since $\way_X(x,-)\leqslant\lambda_X(x)$ and
\begin{align*}
\way_X(x,y)
&=\bigvee_{z\in X}\way_X(x,z)\otimes\way_X(z,y)\\
&=\bigvee_{z\in X}\way_X(x,z)\otimes\bigvee_{\alpha\in FX}(\alpha(y)\tensor [\alpha, \lambda_X(z)])\\
&=\bigvee_{\alpha\in FX}\alpha(y)\tensor\bigvee_{z\in X}(\way_X(x,z)\otimes[\alpha, \lambda_X(z)])\\
&\leqslant\bigvee_{\alpha\in FX}\alpha(y)\tensor[\alpha, \bigvee_{z\in X}\way_X(x,z)\otimes X(z,-)]\\
&=\bigvee_{\alpha\in FX}(\alpha(y)\tensor [\alpha, \way_X(x,-)]).
\end{align*}


\section{The duality}

In this section we assume that a class $\Phi$ of limit weights $\phi:1\modto I$ is given, and we consider the corresponding class $J$ of modules as described in Example \ref{ex:JdefinedByPhi}. Furthermore, let $X$ be a $J$-cocomplete, $J$-continuous and open module determined $\V$-category.\\

\noindent Each $x\in X$ defines:
\begin{align*}
\ev_x\colon FX &\to \V\\
\alpha &\mapsto \alpha(x).
\end{align*}

\begin{lemma}
\label{lemma:ev}
For any $x\in X$, the map $\ev_x$ is an open module on $FX$.
\end{lemma}

\begin{proof}
Certainly, $\ev_x$ is $J$-continuous, since it is the restriction of
\begin{align*}
-\cdot x_*\colon J(X^\op)\to\V &&
\text{(here $x\in X^\op$ and therefore $x_*\colon 1\modto X^\op$)}
\end{align*}
to $FX$. We show now that $\ev_x\in J(FX^\op)$, that is,
\[
 C_x:=\ev_x\cdot-\colon\Mod{\V}(FX,1)\to\V,\,\Psi\mapsto\bigvee_{\alpha\in FX}\Psi(\alpha)\otimes\alpha(x)
\]
preserves $\Phi$-weighted limits. Note that $\Mod{\V}(FX,1)\cong\Mod{\V}(1,FX^\op)$. Furthermore, since $\alpha\in FX$ is $J$-cocontinuous, $C_x=\bigvee_{y\in X}C_y\otimes\way_X(y,x)$. Let $\phi:1\modto I$ be in $\Phi$ and $\Psi_{-}\colon I\to \Mod{\V}(FX,1),\,i\mapsto\Psi_i$ be a $\V$-functor. Then
\begin{align*}
\bigwedge_{i\in I}\V(\phi(i),C_x(\Psi_i))
&=\bigwedge_{i\in I}\V(\phi(i),\bigvee_{y\in X}C_y(\Psi_i)\otimes\way_X(y,x))\\
&= \bigvee_{y\in X}\left(\bigwedge_{i\in I}\V(\phi(i),C_y(\Psi_i))\right)\otimes\way_X(y,x) &\text{($\way(-,x)$ is in $J$)}\\
&\leqslant \bigvee_{\alpha\in FX}\alpha(x)\otimes\bigvee_{y\in X}\bigwedge_{i\in I}\V(\phi(i),C_y(\Psi_i)\otimes[\alpha,\lambda_X y])\\
&\leqslant \bigvee_{\alpha\in FX}\alpha(x)\otimes\bigwedge_{i\in I}\V(\phi(i),\Psi_i(\alpha))
\end{align*}
since
\[C_y(\Psi_i)\otimes[\alpha,\lambda_X y]=\bigvee_{\beta\in FX}\Psi_i(\beta)\otimes[\alpha,\lambda_X y]\otimes[\lambda_X y,\beta]\leqslant\bigvee_{\beta\in FX}\Psi_i(\beta)\otimes[\alpha,\beta]=\Psi_i(\alpha).\qedhere\]
\end{proof}

\noindent We further obtain a map $\eta_X\colon X \to FFX$ given by:
\begin{align}
\label{eq:eta}
x &\mapsto \ev_x.
\end{align}
This is indeed a $\V$-functor, since for any $y,z\in X$ we have:
\[[\eta_X(y),\eta_X(z)] = \bigwedge_{\alpha\in FX}\V(\alpha(y),\alpha(z))\geqslant X(y,z).\]

\begin{lemma}
\label{lemma:FX continuous}
$FX$ is $J$-continuous with the way-below module $\way_{FX}\colon FX\modto FX$ given by:
\begin{equation}
\label{eq:way on FX}
\way_{FX}(\beta,\alpha) = \bigvee_{x\in X} (\alpha(x)\tensor [\beta,\lambda_X(x)]).
\end{equation}
\end{lemma}

\begin{proof}
Note that \eqref{eq:way on FX} states that the way-below module on $FX$ is the restriction of the way-below module on $J(X^\op)$ (see \eqref{way_for_JX}). First we wish to show that
\[\way_{FX}(-,\alpha) := \bigvee_{x\in X}(\alpha(x)\otimes [-,\lambda_X(x)])\]
is a $J$-module of type $FX\modto 1$, for every $\alpha\in FX$. To this end, we consider a diagram
\[1\stackrel{\phi}{\modto} A\stackrel{h}{\to} \V^{FX}\]
where $\phi$ belongs to $\Phi$. We calculate:
\begin{align*}
&\bigwedge_{a\in A}\V(\phi(a),\bigvee_{\beta\in FX}(\way_{FX}(\beta,\alpha)\otimes h(a,\beta)))\\
&= \bigwedge_{a\in A}\V(\phi(a),\bigvee_{x\in X}(\alpha(x)\otimes (\bigvee_{\beta\in FX}([\beta,\lambda_X(x)]\otimes h(a,\beta)))))\\
&\{ \mathrm{put}\ k(a,x) := \bigvee_{\beta\in FX}([\beta,\lambda_X(x)]\otimes h(a,\beta))\ \mathrm{where}\ k\colon A\to\V^{X^\op}\}\\
&=\bigvee_{x\in X}(\alpha(x)\otimes \bigwedge_{a\in A}(\V(\phi(a),k(a,x))))\\
&=\bigvee_{x,y\in X}((\alpha(y)\otimes \way_X(y,x))\otimes \bigwedge_{a\in A}(\V(\phi(a),k(a,x))))\\
&=\bigvee_{\gamma\in FX}\bigvee_{x,y\in X}((\gamma(x)\otimes\alpha(y)\otimes [\gamma,\lambda_X(y)])\otimes \bigwedge_{a\in A}(\V(\phi(a),k(a,x))))\\
&=\bigvee_{\gamma\in FX}\bigvee_{y\in X}(\alpha(y)\otimes [\gamma,\lambda_X(y)]\otimes (\bigvee_{x\in X}(\gamma(x)\otimes \bigwedge_{a\in A}(\V(\phi(a),k(a,x))))))\\
&=\bigvee_{\gamma\in FX}(\way_{FX}(\gamma,\alpha)\otimes \bigwedge_{a\in A}(\V(\phi(a),\bigvee_{x\in X}(\gamma(x)\otimes k(a,x)))))\\
&=\bigvee_{\gamma\in FX}(\way_{FX}(\gamma,\alpha)\otimes \bigwedge_{a\in A}(\V(\phi(a),\bigvee_{\beta\in FX}\bigvee_{x\in X}(\gamma(x)\otimes [\beta,\lambda_X(x)]\otimes h(a,\beta)))))\\
&=\bigvee_{\gamma\in FX}(\way_{FX}(\gamma,\alpha)\otimes \bigwedge_{a\in A}(\V(\phi(a),\bigvee_{\beta\in FX}([\beta,\gamma]\otimes h(a,\beta)))))\\
&\leqslant\bigvee_{\gamma\in FX}(\way_{FX}(\gamma,\alpha)\otimes \bigwedge_{a\in A}(\V(\phi(a),h(a,\beta)))),\\
\end{align*}
as required (recall that the other inequality we get for free).
Furthermore, we calculate:
\begin{align*}
\Sup_{FX}(\way_{FX}(-,\alpha))(x) &= \bigvee_{\beta\in FX}(\way_{FX}(\beta,\alpha)\tensor \beta(x))\\
                                  &= \bigvee_{\beta\in FX}\bigvee_{y\in X}(\alpha(y)\tensor [\beta,\lambda_X(y)]\tensor \beta(x))\\
                                  &= \bigvee_{y\in X}(\alpha(y)\tensor \bigvee_{\beta\in FX}([\beta,\lambda_X(y)]\tensor \beta(x)))\\
                                  &= \bigvee_{y\in X}(\alpha(y)\tensor \bigvee_{\beta\in FX}([\beta,\lambda_X(y)]\tensor [\lambda_X(x),\beta]))\\
                                  &= \bigvee_{y\in X}(\alpha(y)\tensor \way_X(y,x))\\
                                  &= \alpha(x),
\end{align*}
hence $\Sup_{FX}(\way_{FX}(-,\alpha)) = \alpha$. Finally, to conclude that $\mate{\way_{FX}}\dashv\yoneda_{FX}$, let $\psi:FX\modto 1$ in $J$. Let $i$ denote the inclusion $\V$-functor $FX\hookrightarrow J(X^\op)$ and $\way_{J(X^\op)}$ the way-below module on $J(X^\op)$. We observed already that $\way_{FX}=i^*\cdot\way_{J(X^\op)}\cdot i_*$. Hence,
\begin{multline*}
\mate{\way_{FX}}\cdot\Sup_{FX}(\psi)
=(\Sup_{FX}(\psi))^*\cdot\way_{FX}
=(\Sup_{FX}(\psi))^*\cdot i^*\cdot\way_{J(X^\op)}\cdot i_*\\
=(\Sup_{J(X^\op)}(\psi\cdot i^*))^*\cdot\way_{J(X^\op)}\cdot i_*
\leqslant\psi\cdot i^*\cdot i_*=\psi.\qedhere
\end{multline*}
\end{proof}

\begin{lemma}
\label{lemma:FX open module determined}
$FX$ is open module determined.
\end{lemma}

\begin{proof}
For all $\alpha,\beta\in FX$:
\begin{multline*}
\way_{FX}(\beta,\alpha)
=\bigvee_{z\in X}(\alpha(z)\tensor [\beta,\lambda_X(z)])
=\bigvee_{z\in X}(\ev_z(\alpha)\tensor [\lambda_X(z)_*,\beta_*])\\
=\bigvee_{z\in X}(\ev_z(\alpha)\tensor [\ev_z,\lambda_{FX}(\beta)])
=\bigvee_{\mathcal{A}\in FFX}(\mathcal{A}(\alpha)\tensor [\mathcal{A},\lambda_{FX}(\beta)])\qedhere
\end{multline*}
\end{proof}

\noindent By the discussion in Section~\ref{subsect:open modules} and Lemmata~\ref{lemma:FX continuous},~\ref{lemma:FX open module determined} we obtain:
\begin{theorem}
\label{thm:FX}
If $X$ is a $J$-continuous, $J$-cocomplete and open module determined $\V$-category, then so is $FX$.
\end{theorem}

\noindent Our next aim is to show that $\eta_X\colon:X\to FFX$ is an isomorphism. To do so, let now $\mathcal{A}\colon FX\to\V$ be an open module on $FX$. We define:
\[\psi_{\mathcal{A}}(x) := \bigvee_{\alpha\in FX}(\mathcal{A}(\alpha)\tensor[\alpha,\lambda_X (x)]).\]
Such defined $\psi_{\mathcal{A}}$ is a module $X\modto 1$, since it is the composite:
\[X\stackrel{{\lambda_X}_*}{\modto} J(X^\op)^\op\stackrel{i^*}{\modto}FX^\op\stackrel{\mathcal{A}}{\modto}1.\]
We also need to have:
\begin{lemma}
\label{lemma:psi}
For every $\mathcal{A}\in FFX$, we have $\psi_{\mathcal{A}}\in JX$.
\end{lemma}

\begin{proof}
In order to check that $\psi_{\mathcal{A}}\colon X\modto 1$ belongs to $JX$, we need to check whether $\psi_{\mathcal{A}}\cdot -\colon \V^X\to \V$ preserves $\Phi$-weighted limits. Let
\[1\stackrel{\phi}{\modto} A\stackrel{h}{\to} \V^X\]
be a limit diagram with $\phi$ in $\Phi$. Spelled out, we have to show that
\begin{equation*}
\bigvee_{x\in X}(\psi_{\mathcal{A}}(x)\tensor \bigwedge_{y\in A}(\V(\phi(y),h(y,x))))\geqslant \bigwedge_{y\in A}(\V(\phi(y),\bigvee_{x\in X}(\psi_{\mathcal{A}}(x)\tensor h(y,x)))).
\end{equation*}
To this end, we calculate:
\begin{align*}
 & \bigwedge_{y\in A}(\V(\phi(y),\bigvee_{x\in X}(\psi_{\mathcal{A}}(x)\tensor h(y,x))))\\
 &= \bigwedge_{y\in A}(\V(\phi(y),\bigvee_{x\in X}\bigvee_{\alpha\in FX}(\mathcal{A}(\alpha)\tensor [\alpha, \lambda_X(x)]\tensor h(y,x))))\\
 &= \bigwedge_{y\in A}(\V(\phi(y),\bigvee_{\alpha\in FX}(\mathcal{A}(\alpha)\tensor \way_{FX}(\alpha,h(y))))) \ \ \ \ \{ \mathrm{since}\ \mathcal{A}^\op\in J\}\\
 &= \bigvee_{\alpha\in FX}(\mathcal{A}(\alpha)\tensor \bigwedge_{y\in A}(\V(\phi(y), \way_{FX}(\alpha,h(y)))))\\
 &= \bigvee_{\alpha,\beta\in FX}((\mathcal{A}(\beta)\tensor \way_{FX}(\beta,\alpha))\tensor \bigwedge_{y\in A}(\V(\phi(y), \way_{FX}(\alpha,h(y)))))\\
 &= \bigvee_{\alpha,\beta\in FX}\bigvee_{x\in X}((\mathcal{A}(\beta)\tensor \alpha(x)\tensor [\beta,\lambda_X(x)])\tensor \bigwedge_{y\in A}(\V(\phi(y), \way_{FX}(\alpha,h(y)))))\\
 &= \bigvee_{x\in X}\bigvee_{\beta\in FX}(\mathcal{A}(\beta)\tensor [\beta,\lambda_X(x)])\tensor \bigvee_{\alpha\in FX} \ev_x(\alpha)\tensor \bigwedge_{y\in A}(\V(\phi(y), \way_{FX}(\alpha,h(y))))\\
 &\{\ev_x\ \mathrm{is\ a\ filter}\}\\
 &= \bigvee_{x\in X} (\psi_{\mathcal{A}}(x) \tensor \bigwedge_{y\in X}\V(\phi(y),\bigvee_{\alpha\in FX}(\alpha(x)\tensor \way_{FX}(\alpha, h(y)))))\\
 &\leqslant \bigvee_{x\in X} (\psi_{\mathcal{A}}(x) \tensor \bigwedge_{y\in X}\V(\phi(y),\alpha(x)\tensor [\alpha, h(y)]))\\
 &\leqslant \bigvee_{x\in X} (\psi_{\mathcal{A}}(x) \tensor \bigwedge_{y\in X}\V(\phi(y), h(y,x))),
\end{align*}
which proves $\psi_{\mathcal{A}}\in JX$.
\end{proof}

\begin{lemma}
\label{lemma:AB}
For any $\alpha\in FX$, we have $\mathcal{A}(\alpha) = \alpha(\Sup_X(\psi_{\mathcal{A}}))$.
\end{lemma}

\begin{proof}
\begin{align*}
\alpha(\Sup_X(\psi_{\mathcal{A}})) &= \colim(\alpha,\psi_{\mathcal{A}})\\
                                   &= \bigvee_{x\in X}(\alpha(x)\tensor \psi_{\mathcal{A}}(x))\\
                                   &= \bigvee_{x\in X}(\alpha(x)\tensor \bigvee_{\beta\in FX}(\mathcal{A}(\beta)\tensor [\beta,\lambda_X(x))))\\
                                   &= \bigvee_{\beta\in FX}(\mathcal{A}(\beta)\tensor \bigvee_{x\in X}(\alpha(x)\tensor [\beta,\lambda_X(x)]))\\
                                   &= \bigvee_{\beta\in FX}(\mathcal{A}(\beta)\tensor \way_{FX}(\beta,\alpha))\\
                                   &= \colim(\mathcal{A},\way_{FX}(-,\alpha))\\
                                   &= \mathcal{A}(\Sup_{FX}(\way_{FX}(-,\alpha)))\\
                                   &= \mathcal{A}(\alpha).\qedhere
\end{align*}
\end{proof}

\begin{definition}
We say that a $\V$-functor $f\colon X\to Y$ between $\V$-categories {\em reflects open modules} if $\alpha \cdot f\in FX$ for every $\alpha\in F{Y}$. Let $\QDOM$ be the category of $J$-cocomplete, $J$-continuous and open module determined $\V$-categories together with open module reflecting maps.
\end{definition}

\begin{lemma}
The pair of operations
\begin{align*}
X &\ \ \mapsto\ \ FX \\
f\colon X\to Y &\ \ \mapsto\ \ -\cdot f\colon FY\to FX
\end{align*}
defines a contravariant functor, i.e.  $F\colon \QDOM^\op\to\QDOM$.
\end{lemma}

\begin{proof}
Functoriality is trivial; we only need to show that $F(f)$ reflect open modules.
Let $\mathcal{A}\in F{F{X}}$. By Lemma~\ref{lemma:AB} there exists $x\in X$ such that $\mathcal{A} =\ev_{x}$, namely $x= \Sup_{X}\psi_{\mathcal{A}}$. Then, for any $\alpha\in F{Y}$, we have $(\mathcal{A} \cdot F(f))(\alpha) = \mathcal{A}(\alpha \cdot f) = \alpha(f(x))= \ev_{f(x)}(\alpha)$. Hence $\mathcal{A} \cdot F(f) = \ev_{f(x)}$, i.e. $\mathcal{A} \cdot F(f)\in FFX$.
\end{proof}

\begin{theorem}[The Duality Theorem]\label{thm:duality}
The category $\QDOM$ is self-dual.
\end{theorem}

\begin{proof}
The natural isomorphism $\eta\colon 1_{\QDOMindex}\to FF$ as defined in (\ref{eq:eta}) has the converse $\eps \colon FF\to 1_{\QDOMindex}$ given by $\eps_X(\mathcal{A})= \Sup_X\psi_{\mathcal{A}}$ for every $\mathcal{A}\in F{F{X}}$.
\end{proof}

\section{Examples of the duality}

\subsection{Lawson duality}\label{subsect:4.1} The case $\V=\two$ and $J=\FSW$, perhaps the simplest possible, served us as a proof guide throughout the paper. In fact, most of the crucial proof ideas (e.g. Lemma~\ref{lemma:AB}: any open module on open modules $\mathcal{A}$ is of the form $\ev_{\Sup_X \psi_{\mathcal{A}}}$ for some $J$-ideal $\psi_{\mathcal{A}}$) come from an analysis of this simple case. Observe that $\FSW$-continuous, $\FSW$-cocomplete $\two$-categories are continuous dcpos (domains). Furthermore, open modules are nothing else but (the characteristic maps of) Scott-open filters on domains. Recall that in this case any $FX$ is open module determined: the equality~(\ref{eq:technical lemma}) reduces to
\[\forall x,y\in X\ (x\wb y\ \implies\ \exists \alpha\in FX\ (y\in\alpha\pzb \uup x)),\]
and we define such $\alpha\in FX$ by $\alpha := \bigcup_{n\in\omega} \uup x_n$, where the descending chain $(x_n)_{n\in\omega}$ has been obtained by a repeated use of interpolation (see Prop.~3.3 of \cite{compendium}): \[x\wb\ldots\wb x_n\wb x_{n-1}\wb\ldots\wb x_2\wb x_1\wb x_0=y.\]
Consequently, the category $(\FSW,\two)\!\!-\!\!\mathbf{Dom}$ is the category of domains with open filter reflecting maps; our Theorem~\ref{thm:duality} reduces to Theorem~IV-2.12 of \cite{compendium} establishing the Lawson duality for domains. It is worth mentioning that the Lawson duality (originally proved in \cite{law79}) finds its applications in the theory of locally compact spaces; in particular, the lattice of opens of a locally compact sober space $X$ is Lawson dual to the lattice of compact saturated subsets of $X$ (cf. Hofmann-Mislove theorem).

\subsection{A metric duality}\label{subsect:4.2} In the case $\V=[0,\infty]$ with $\otimes=+$ and $J$ being the class of $\FSW$-ideals (or, equivalently, flat modules), our duality works in a~certain subcategory of $\Met$: its $\FSW$-cocomplete objects are known in the literature as Yoneda-complete gmses \cite{rutten:yoneda}. The $\FSW$-cocomplete and $\FSW$-continuous ones form a class not previously discussed in the literature, except in the forthcoming paper \cite{KW:balls}, where they are shown to be precisely the spaces having continuous and directed-complete formal ball models \cite{EH98, aliakbari, RV09} (this implies, in particular, that their topology and metric structure can be respectively characterized as a subspace Scott topology and a partial metric on a domain).

A proof that objects of $(\FSW,[0,\infty])\!\!-\!\!\mathbf{Dom}$ are open filter determined can be found in \cite{sylwia}; below we present a sketch of the proof. 

We abbreviate $\way_X$ to $\way$ and customarily use $+$ instead of $\otimes$, $\inf$ instead of $\bv$, etc. In order to show (\ref{eq:technical lemma}) 
it is enough to find a family of open filters $(\alpha_{ e, b})_{ e, b> 0}$,
such that $ e>\way(x,y)$ implies \[ e+ b \geqslant \alpha_{ e, b}(y) + [\alpha_{ e, b}, \way(x,-)] \geqslant
\inf_{\alpha\in FX} (\alpha(y) + [\alpha, \way(x,-)]),\] which, by complete distributivity of $([0,\infty],\geqslant)$, allows us to draw the desired conclusion. Take an arbitrary $ e>\way(x,y)$ and $ b>0$, and choose a chain
$( e_n)_{n\in\omega}$ in $([0,\infty],\geqslant)$ such that:
\begin{equation}
\label{equation:epsy}
\begin{split}
&  b> e_0+ e_0,\\
&  e_0 >  e_1 >  e_2 > \ldots >  e_n > \ldots > 0,\\
&  e_n \geqslant  e_{n+1} +  e_{n+2} + \ldots,\\
& \inf_{n\in \omega} e_n = 0.	
\end{split}
\end{equation}
Now, by interpolation, we can find a sequence $(x_n)_{n\in\omega}$ such that:
\begin{align*}
 &  e > \way(x,x_0) + \way(x_0,y) & \textmd{and}\  &  e_0 > \way(x_0,y),\\
 &  e > \way(x,x_1) + \way(x_1, x_0) + \way(x_0,y), & \textmd{and}\ &  e_1 > \way(x_1,x_0),\\
 &  e > \way(x,x_2) + \way(x_2,x_1) + \way(x_1, x_0) + \way(x_0,y) & \textmd{and}\ &  e_2 > \way(x_2,x_1),\\
 & \cdots & &\\
 &  e > \way(x,x_n) + \way(x_n,x_{n-1}) + \dots + \way(x_1, x_0)
+ \way(x_0,y) & \textmd{and}\ &  e_n > \way(x_n,x_{n-1}),\\
 & \cdots & &
\end{align*}
Define $\alpha_{ e, b}\colon X\to [0,\infty]$ as $\alpha_{ e, b}(z) := \inf_{n\in\omega}\sup_{k\geq n}X(x_k,z)$; this map is
an open module on $X$. In order to conclude (\ref{eq:technical lemma}), it is now enough to verify that
\begin{equation}\label{eq:pom2}
 e +  b \geqslant \alpha_{ e, b}(y) + [\alpha_{ e, b},\way(x,-)].
\end{equation}
However
\begin{eqnarray*}
\alpha_{ e, b}(y)  & = & \inf_{n\in\omega}\sup_{k\geq n}X(x_k,y)\\
                         & \leqslant & \sup_{k\geq 1}\ (X(x_k,x_{k-1}) + \dots + X(x_1,x_0) + X(x_0,y))\\
                         & \leqslant & \sup_{k\geq 1}\ (\way(x_k,x_{k-1}) + \dots + \way(x_1,x_0) + \way(x_0,y)) \ \ \  \{\mathrm{by}\ (\ref{equation:epsy})\}\\
                         & \leqslant &  e_0+ e_0\\
                         & < &  b.
\end{eqnarray*}
and
\begin{eqnarray*}
    [\alpha_{ e, b},\way(x,-)] & = & \sup_{z\in X}(\mmminus{\alpha_{e, b}(z)}{\way(x,z)})\\
      & \leqslant & \sup_{z\in X}(\mmminus{(\inf_{n\in\omega}\sup_{k\geq n} X(x_k,z))}{\way(x,z)})\\
      & \leqslant & \sup_{z\in X}(\inf_{n\in\omega}\sup_{k\geq n}(\mmminus{X(x_k,z)}{\way(x,z)}))\\
      & \leqslant & \sup_{n\in\omega}\sup_{k\geq n} \way(x,x_k)\\
      & \leqslant &  e.
\end{eqnarray*}
so (\ref{eq:pom2}), and therefore also (\ref{eq:technical lemma}) are now verified.

\subsection{An ultrametric duality}\label{subsect:ultramet}

For the quantale $\V=[0,\infty]$ with $\otimes=\max$, $\Cat{\V}$ is the category $\UMet$ of ultrametric spaces and contraction maps. As above, we can choose $J$ to be the class of all flat modules (see Example \ref{example:vickers}), and obtain that the corresponding category $\QDOM$ is self-dual. However, in ultrametric spaces flat modules are not, in general, $\FSW$-ideals, as the following example shows. 

\begin{example}
Consider the set $\Nat$ of natural numbers with the distance
\[\Nat(n,m)=\begin{cases} 
0 & \text{if $n=m$,}\\
\max(n,m) &\text{otherwise.}
\end{cases}
\]
This distance is a symmetric, separable ultrametric. 
Take
\[\phi(x)=\begin{cases} 
0 & \text{if $x=0$,}\\
1 &\text{if $x>0$.}
\end{cases}
\]
Trivially, $\phi$ preserves the empty meet. Now, observe that the proof of (the equivalence of (1) and (2) of) Proposition 7.9 in 
\cite{vickers} holds verbatim for $\otimes = \max$, hence
it is enough to show that $(\phi \cdot -)$ preserves meets of modules of the form $\max(\Nat(-,x),c)$ for some $c\in [0,\infty]$. Suppose $A:= \max(\Nat(-,a),c_1)$ and 
$B:= \max(\Nat(-,b),c_2)$ for $c_1,c_2\in[0,\infty]$; we are heading to prove:
 \noindent
\begin{equation}\label{eq:111}\tag{*}
\inf_{z\in \Nat} \max(Az,Bz,\phi z)= \max(\inf_{s\in\Nat} (\max(As,\phi s)),\inf_{r\in\Nat} (\max (Br,\phi r))).
\end{equation}

\noindent
We have
\begin{align*}
& \inf_{z\in \Nat} \max(Az,Bz,\phi z)
= \inf_{z\in \Nat} \max(z,a,b,c_1,c_2,\phi z)
= \max(a,b,c_1,c_2),\\
&\inf_{s\in\Nat} \max(As,\phi s)
= \inf_{s\in\Nat} \max(s,a,c_1,\phi s)
= \max(a,c_1),\\
&\inf_{r\in\Nat} \max(Br,\phi r)
=\inf_{r\in\Nat} \max(r,b,c_2,\phi r)=\max(b,c_2)
\end{align*}
since all these infima are attained for $z=r=s=0$. This shows (\ref{eq:111}), and so $\phi\colon X\modto 1$ is a flat module.

\noindent
On the other hand, $\phi$ is not an $\FSW$-ideal: we have $\phi(2)<2$ and $\phi(3)<2$ but there is no $z\in \Nat$ with $\phi(z) < 1$ and $\Nat(2,z)<2$ and $\Nat(3,z)<2$.
\end{example}

\subsection{The absolute case}

For any quantale $\V$, we can consider $\Phi$ being the empty class and therefore $JX=\widehat{X}$ is the collection 
of all modules of type $X\modto 1$. In this case, every cocontinuous $\V$-functor $\alpha\colon X\to\V$ is an open module. 
Furthermore, every totally continuous cocomplete \mbox{$\V$-category} is open module determined 
since $\way_X(x,-)\colon X\to\V$ is in $FX$. Finally, a $\V$-functor $f\colon X\to Y$ reflects open modules if 
and only if $f$ is left adjoint. Therefore Theorem \ref{thm:duality} states that the category of totally 
continuous cocomplete $\V$-categories and left adjoint $\V$-functors is self-dual.
 
\subsection{A somehow different example}\label{subsect:4.3}

We consider now $\V=[0,\infty]$ where $\otimes=+$, with the class $J$ of modules described in Example \ref{ex:formal balls}. However, for technical reasons we consider the unique module $\varnothing\modto 1$ as a formal ball, so that $J\varnothing=1$. 
Consequently, the empty space is not $J$-cocomplete. We will show now that our duality theorem holds in this case too, despite the fact that this class of modules is (to our knowledge) not defined via a class of limit weights.

Let now $X$ be a $J$-cocomplete and $J$-continuous metric space. We write $\way:X\to JX$ for the left adjoint to $\Sup:JX\to X$. Hence, for any $x\in X$, $\way(x)\in JX$ is of the form $\way(x)=X(-,x_1)+u$ for some $x_1\in X$ and $u\in[0,\infty]$. Note that $u<\infty$ if $x$ is not the bottom element of $X$. Assume that $\way(x_1)=X(-,x_2)+u_2$. Then
\[
 X(-,x_1)+u=\way(x)=\way(x_1+u_1)=\way(x_1)+u_1=X(-,x_2)+u_2+u_1,
\]
hence, $X(-,x_1)=X(-,x_2)+u_2$. In particular, $0=X(x_1,x_2)+u_2$, and therefore $u_2=0$ and we obtain $\way(x_1)=\yoneda(x_1)$. Let $A$ be the equaliser of $\yoneda$ and $\way$, that is, $A=\{x\in X\mid \way(x)=\yoneda(x)\}$. By the considerations above, $\way:X\to JX$ factors through the inclusion $JA\hookrightarrow JX$. Moreover, for any $X(-,x)+u$ with $x\in A$, $\way(x+u)=\way(x)+u=X(-,x)+u$, which gives $X\cong JA$. We also remark that $x\in A$ if and only if $X(x,-):X\to[0,\infty]$ preserves tensoring. One has $\phi\in FX$ precisely if $\phi=X(x,-)+u$ for some $x\in X$ and $u\in[0,\infty]$ and if, moreover, $\phi$ preserves tensoring. If $u<\infty$, then also $X(x,-)$ preserves tensoring, hence $x\in A$. Consequently, $FX\cong J(A^\op)$. 

Consider now $f:X\to Y$ with $X\cong JA$ and $Y\cong JB$ as above. Then $f$ is open module reflecting if, and only if, for each $y_0\in B$, there exists some $x_0\in A$ and some $v\in[0,\infty]$ with $Y(y_0,f(-))=X(x_0,-)+v$. We show that $f$ necessarily preserves tensoring. To this end, let $x\in X$ and $u\in[0,\infty]$. Then
\begin{multline*}
 Y(y_0,f(x+u))=X(x_0,x+u)+v=X(x_0,x)+v+u =Y(y_0,f(x))+u=Y(y_0,f(x)+u)
\end{multline*}
for all $y_0\in B$, hence $f(x+u)=f(x)+u$. Therefore $f$ corresponds to a module $\phi:B\modto A$ in the sense that, when identifying $X$ with $JA$ and $Y$ with $JB$, 
then $f(\psi)=\psi\cdot\phi$. Hence, for any $x\in A$, $x^*\cdot\phi=\phi(-,x)$ belongs to $JB$, and the $f$ being open module reflecting translates 
to $\phi\cdot y_*=\phi(y,-)\in J(A^\op)$ for all $y\in B$. Recall that for each module $\phi:B\modto A$ we have its dual $\phi^\op:A^\op\modto B^\op,\,\phi^\op(x,y)=\phi(y,x)$, 
and with this notation the latter condition reads as $y^*\cdot\phi^\op\in J(A^\op)$ for all $y\in B^\op$. We conclude that the category of $J$-cocomplete and $J$-continuous metric spaces and open 
module reflecting contraction maps is dually equivalent to the category of all metric spaces with morphisms those modules $\phi:X\to Y$ satisfying
\begin{align*}
\forall y\in Y\,.\,(y^*\cdot\phi\in JX) &&\text{and}&&
\forall x\in X^\op\,.\,(x^*\cdot\phi^\op\in J(Y^\op)),
\end{align*}
and the latter category is obviously self-dual.

\end{document}